\documentclass[12pt, reqno]{amsart}

\usepackage{latexsym}
\usepackage{amsmath}
\usepackage{amsfonts}
\usepackage{amssymb}
\usepackage{graphicx}
\usepackage{comment}
\usepackage{enumitem}

\newtheorem{theorem}{Theorem}[section]
\newtheorem{lemma}[theorem]{Lemma}
\newtheorem*{theorem*}{Theorem}
\newtheorem{corollary}[theorem]{Corollary}

\newtheorem* {question*}{Question}

\theoremstyle{definition}

\theoremstyle{remark}

\numberwithin{equation}{section}

\newcommand{\cl}{{\rm cl}}
\newcommand{\Mq}{\mathcal{M}_q}

\begin{document}

\title[$GF(q)$-Chordal Matroids]{Chordal Matroids arising from Generalized parallel connections}

\author{James Dylan Douthitt}
\address{Mathematics Department\\
Louisiana State University\\
Baton Rouge, Louisiana}
\email{jdouth5@lsu.edu}

\author{James Oxley}
\address{Mathematics Department\\
Louisiana State University\\
Baton Rouge, Louisiana}
\email{oxley@math.lsu.edu}

\subjclass{05B35, 05C75}
\date{\today}

\begin{abstract}
A graph is chordal if every cycle of length at least four has a chord. In 1961, Dirac characterized chordal graphs as those graphs that can be built from complete graphs by repeated clique-sums. Generalizing this, we consider the class of simple $GF(q)$-representable matroids that can be built from projective geometries over $GF(q)$ by repeated generalized parallel connections across projective geometries. We show that this class of matroids is closed under induced minors. We characterize the class by its forbidden induced minors; the case when $q=2$ is distinctive.
\end{abstract}

\maketitle

\section{Introduction}
\label{introduction}
The notation and terminology in this paper will follow \cite{diestel} for graphs and \cite{ox1} for matroids. Unless stated otherwise, all graphs and matroids considered here are simple. Thus every contraction of a set from a matroid is immediately followed by the simplification of the resulting matroid. A \textit{chord} of a cycle $C$ in a graph $G$ is an edge $e$ of $G$ that is not in $C$ such that both vertices of $e$  are vertices of $C$. A graph is \textit{chordal} if every cycle of length at least four has a chord. Such graphs were called \textit{rigid circuit graphs} by Dirac \cite{dirac} and \textit{triangulated graphs} by Berge \cite{berge}. Let $G_1$ and $G_2$ be graphs and $V(G_1)\cap V(G_2)=V$, say. Assume that $G_1[V]$ is a complete graph $H$ and $G_2[V]$ has edge set $E(H)$. The \textit{clique-sum} of $G_1$ and $G_2$ is the graph with vertex set $V(G_1)\cup V(G_2)$ and edge set $E(G_1)\cup E(G_2)$. Loosely speaking, the clique-sum is obtained by gluing $G_1$ and $G_2$ together across the clique $H$. While there are several characterizations of chordal graphs (see, for example, \cite{rose}), we choose to focus on the following one of Dirac \cite{dirac}.
\begin{theorem}\label{theo1}
A graph $G$ is chordal if and only if $G$ can be constructed from complete graphs by repeated clique-sums.
\end{theorem}

Let $M_1$ and $M_2$ be matroids whose ground sets intersect in a set $T$ such that $T$ is a modular flat of $M_1$, and $M_1|T=M_2|T=~N$. The \textit{generalized parallel connection} of $M_1$ and $M_2$ across $N$ is the matroid with ground set $E(M_1) \cup E(M_2)$ whose flats are those subsets $X$ of $E(M_1) \cup E(M_2)$ such that $X\cap E(M_1)$ is a flat of $M_1$, and $X\cap E(M_2)$ is a flat of $M_2$. We denote this matroid by $P_N(M_1,M_2)$ or $P_T(M_1,M_2)$. Note that $T$ may be empty, in which case, $P_T(M_1,M_2)=M_1\oplus M_2$. 

For a prime power $q$, we will denote the projective geometry $PG(r-1,q)$ by $P_r$ when context makes the field clear. Let $\mathcal{M}_q$ be the class of matroids that can be built from projective geometries over $GF(q)$ by a sequence of generalized parallel connections across projective geometries over $GF(q)$. A matroid $M$ is \textit{$GF(q)$-chordal} if $M$ is a member of $\mathcal{M}_q$. By \cite{brylawski}, each member of $\mathcal{M}_q$ is $GF(q)$-representable.

An \textit{induced minor of a graph} $G$ is a graph $H$ that can be obtained from $G$ by a sequence of vertex deletions and edge contractions. Similarly, an \textit{induced minor of a matroid} $M$ is a matroid $N$ that can be obtained from $M$ by a sequence of restrictions to flats and contractions, where each such contraction is followed by a simplification. Evidently, the class of chordal graphs is closed under vertex deletions, that is, it is closed under taking induced subgraphs. The analogous property for the class $\mathcal{M}_q$ is highlighted by the following result. 

\begin{theorem}\label{theo2}
    For all $q$, the class $\mathcal{M}_q$ is closed under taking induced minors.
\end{theorem}
Our main results are the following characterizations of the forbidden induced minors for the class $\mathcal{M}_q$, first for $q=2$ and then for $q>2$.
\begin{theorem}\label{theo3}
    The set of forbidden induced minors for the class $\mathcal{M}_2$ is $\{M(K_4),U_{3,4}\}$.
\end{theorem}

\begin{theorem}\label{theo4}
    For each $q>2$, the set of forbidden induced minors for the class $\mathcal{M}_q$ is $\{U_{2,k}: 2<k\leq q\}\cup \{U_{3,q+2}\}$.
\end{theorem}

Several different notions of chordal matroids have been given over the last forty years \cite{barahona, bonin, cordovil, mayhew, ziegler}. Each of these papers, with the exception of \cite{cordovil, mayhew}, focuses primarily on binary matroids. Following Cordovil, Forge, and Klein \cite{cordovil}, we define a simple or non-simple matroid $M$ to be \textit{chordal} if, for each circuit $C$ that has at least four elements, there are circuits $C_1$ and $C_2$ and an element $e$ such that $C_1\cap C_2 =\{e\}$ and $C=(C_1 \cup C_2)-e$.

Cordovil, Forge, and Klein \cite{cordovil} and Mayhew and Probert \cite{mayhew} study the relation between chordal graphs and supersolvable matroids. Such matroids were originally introduced by Stanley \cite{stanley}. A rank-$r$ matroid is \textit{supersolvable} if there is a chain of modular flats $F_1\subseteq F_2\subseteq \dots \subseteq F_r$ in $M$ where the rank of $F_i$ is $i$ for each $i$ in $[r]$. The class $\mathcal{M}_q$ is given in \cite{mayhew} as an example of a class of matroids whose members are supersolvable and form a saturated class of matroids, where a matroid $M$ is \textit{saturated} if, for every modular flat $F$, the restriction $M|F$ has no two disjoint cocircuits.

\section{Preliminaries}
Before beginning the discussions of forbidden induced minors for the class $\mathcal{M}_q$, we show that this class is closed under taking induced minors. We shall use the following well-known property of generalized parallel connections (see, for example, \cite[11.23]{ox1}).

\begin{lemma}\label{lem1} For every flat $F$ of $P_N(M_1,M_2)$,
\[
r(F)=r(F\cap E(M_1))+r(F\cap E(M_2))-r(F\cap E(N)).
\]
\end{lemma}

\begin{lemma}\label{lem2} The class $\mathcal{M}_q$ of $GF(q)$-chordal matroids is closed under induced restrictions.
\end{lemma}

\begin{proof}
It is enough to show that the class $\mathcal{M}_q$ is closed under restricting to a hyperplane. Let $M$ be a minimum-rank matroid in $\mathcal{M}_q$ such that $M|H\not\in \mathcal{M}_q$ for some hyperplane $H$ of $M$. Then $M$ is not a projective geometry, so $M=P_N(M_1,M_2)$ where $M_1$ and $M_2$ are in $\mathcal{M}_q$, and $N$ is a projective geometry over $GF(q)$. Let $H_i=H\cap E(M_i)$ for each $i$ in $\{1,2\}$ and let $H_N=H\cap E(N)$. Then, since $H$ is a both flat of a generalized parallel connection and a hyperplane of $M$,
\begin{equation}
r(M_1)+r(M_2)-r(N)-1=r(H)=r(H_1)+r(H_2)-r(H_N).
\end{equation}

Suppose $H$ contains all of $E(N)$. Then $r(H_N)=r(N)$, so, from $(2.1)$, we have
\[
r(M_1)+r(M_2)-1=r(H_1)+r(H_2).
\]
This implies that, for some $i$ and $j$ such that $\{i,j\}=\{1,2\}$, the hyperplane $H$ contains $E(M_i)$, and $H_j$ is a hyperplane of $M_j$. It follows by the minimality of $M$ that $M|H$ is in $\mathcal{M}_q$, a contradiction.

We now assume that $H$ does not contain $E(N)$. Hence $H$ does not contain $E(M_1)$ or $E(M_2)$. By \cite{brylawski}, $E(M_1)$ is a modular flat of $P_N(M_1,M_2)$, so 
\begin{equation}
    r(H)=r(H\cup E(M_1))+r(H_1)-r(M_1).
\end{equation}

As $E(M_1)\not \subseteq H$, we deduce that $r(H\cup E(M_1))>r(H)$, so $r(H\cup E(M_1))=r(H)+1$. It follows from $(2.2)$ that $r(H_1)=r(M_1)-1$. By symmetry, $r(H_2)=r(M_2)-1$. Hence $H_i$ is a hyperplane of $M_i$ for each $i$ in $\{1,2\}$. Therefore, by $(2.1)$ and $(2.2)$,
\begin{flalign*}
r(M_1)+r(M_2)-r(N)-1= r(M_1)-1+r(M_2)-1-r(H_N).
\end{flalign*}
Thus $r(H_N)=r(N)-1$. Hence $H_N$ is a hyperplane of $N$, so $M|H_N$ is a projective geometry. By \cite{brylawski} (see also \cite[Proposition 11.4.15]{ox1}), $M|H=P_{M|H_N}(M|H_1,M|H_2)$, so $M|H$ is a generalized parallel connection of members of $\mathcal{M}_q$ along a projective geometry and is therefore a member of $\mathcal{M}_q$, a contradiction. 
\end{proof}

In the next result, where we show that $\mathcal{M}_q$ is closed under contractions, we shall make repeated use of the fact that every contraction is followed by a simplification. We observe that a consequence of this result is that $\Mq$ is closed under parallel minors.

\begin{lemma}\label{lem3} The class $\mathcal{M}_q$ of $GF(q)$-chordal matroids is closed under taking contractions. 
\end{lemma}

\begin{proof} Let $M$ be a $GF(q)$-chordal matroid. We argue by induction on $|E(M)|$ that $M/e$ is a $GF(q)$-chordal matroid for all $e$ in $E(M)$. The result is certainly true if $|E(M)|\leq 2$. Suppose the result holds when $|E(M)|<k$ and let $|E(M)|=k$. The result holds if $M$ is a projective geometry, so we may assume that $M=P_N(M_1,M_2)$, where $M_1$ and $M_2$ are in $\mathcal{M}_q$, and $N$ is a projective geometry over $GF(q)$. Suppose that $e\in E(M_1)-E(N)$. Then $M_1/e$ is in $\Mq$ by the induction hypothesis and $M_1/e$ has $N$ as a restriction. As $M/e=P_N(M_1/e,M_2)$, we deduce that $M/e$ is in $\mathcal{M}_q$. We may now assume that $e\in E(N)$. This implies that $e$ is in both $E(M_1)$ and $E(M_2)$, and, by the induction hypothesis, $M_1/e$ and $M_2/e$ are in $\mathcal{M}_q$ and contain the projective geometry $N/e$ as a restriction. By \cite{brylawski}, $M/e=P_{N/e}(M_1/e,M_2/e)$, and we again get that $M/e$ is in $\mathcal{M}_q$.
\end{proof}

\begin{proof}[Proof of Theorem \ref{theo2}] This is an immediate consequence of combining  Lemmas~\ref{lem2}~and~\ref{lem3}.
\end{proof}

\section{Characterizing $GF(q)$-chordal matroids}

In this section, we prove Theorems \ref{theo3} and \ref{theo4}. An \textit{induced-minor-minimal non-$GF(q)$-chordal matroid} is a $GF(q)$-representable matroid that is not a $GF(q)$-chordal matroid such that every proper induced minor of $M$ is a $GF(q)$-chordal matroid. When $q=2$, the only rank-2 binary matroids are $U_{2,2}$ and $U_{2,3}$ both of which are $GF(2)$-chordal. Clearly, neither $M(K_4)$ nor $U_{3,4}$ is $GF(2)$-chordal. It follows that each is an induced-minor-minimal non-$GF(2)$-chordal matroid. When $q>2$, the matroids in $\{U_{2,3}, U_{2,4},\dots, U_{2,q}\}$ are induced-minor-minimal non-$GF(q)$-chordal matroids.

For the remainder of this paper, it will be convenient to view a $GF(q)$-representable matroid $M$ of rank $r$ as a restriction of $P_r$ by coloring the elements of $E(M)$ \textit{green} and coloring the other elements \textit{red}. Note that when we contract an element $e$ from $M$, we can obtain $M/e$ as follows. Take a hyperplane $H$ of $P_r$ that avoids $e$. Then project from $e$ onto $H$. Because $e$ is green, an element $z$ of $H$ is green in the contraction precisely when there are at least two green points on the line $\cl_{P_r}(\{e,z\})$ of $P_r$. 

For a positive integer $k$, a partition $(X,Y)$ of the ground set of a matroid $M$ is a \textit{vertical $k$-separation} of $M$ if $r(X)+r(Y)-r(M)\leq k-1$ and $\min \{r(X),r(Y)\} \geq k$. This vertical $k$-separation is \textit{exact} if $r(X)+r(Y)-r(M)=k-1$.

\begin{lemma}\label{lem4}
    Let $(X,Y)$ be a vertical $2$-separation in a matroid $M$ such that each of $M|\cl (X)$ and $M|\cl(Y)$ is in $\Mq$. Then either
    \begin{enumerate}[label=(\roman*)]
    \item $|\cl(X)\cap \cl(Y)|=1$ and $M\in \Mq$; or 
    \item $|\cl(X)\cap \cl(Y)|=0$ and $M$ has $U_{3,4}$ and $U_{2,3}$ as induced minors.
    \end{enumerate}
\end{lemma}

\begin{proof}
    Observe that if $|\cl(X)\cap \cl(Y)|=1$, then $M$ is the parallel connection of $M|\cl(X)$ and $M|\cl(Y)$, so $M\in \Mq$. Now suppose that $\cl(X)\cap \cl(Y)=\emptyset$. Let $C$ and $D$ be circuits of $M$ each of which meets both $X$ and $Y$ such that $|C\cap X|$ is a minimum and $|D\cap Y|$ is a minimum. Because $M$ can be written as a $2$-sum with basepoint $b$ of matroids with ground sets $X\cup b$ and $Y\cup b$, it follows that $(C\cap X)\cup (D\cap Y)$ is a circuit of $M$. Clearly both $M|\cl(C\cap X)$ and $M|\cl(D\cap Y)$ are in $\Mq$. Let $X_1$ and $Y_1$ be subsets of $C\cap X$ and $D\cap Y$, respectively, such that $|X_1|=|C\cap X|-2$ and $|Y_1|=|D\cap Y|-2$. Then each of $(M|\cl(C\cap X))/X_1$ and $(M|\cl(D\cap Y))/Y_1$ is a rank-2 matroid in $\Mq$. Moreover, in $M/X_1$, there is no element $x$ of $X$ that is in the closure of $Y$ otherwise $(X_1\cup x)\cup (D\cap Y)$ contains a circuit of $M$ that contains $x\cup (D\cap Y)$ and violates the choice of $C$. Hence the rank-2 matroid $(M|\cl(C\cap X))/X_1$, which is in $\Mq$, is not isomorphic to $U_{2,q+1}$. Thus it is isomorphic to $U_{2,2}$.
    % Since $M|\cl(C\cap X)\in \Mq$, it follows that $(M|\cl(C\cap X))/X_1\cong U_{2,2}$. 
    By symmetry, $(M|\cl(D\cap Y))/Y_1\cong U_{2,2}$. Hence $(M|\cl((C\cap X)\cup (D\cap Y)))/(X_1\cup Y_1)\cong U_{3,4}$. Thus both $U_{3,4}$ and $U_{2,3}$ are induced minors of $M$.
\end{proof}

In the next result, we denote by $P_{r+1}\backslash P_{r-i}$ the matroid that is obtained from $P_{r+1}$ by deleting the elements of a rank-$(r-i)$ flat. Clearly this matroid does not depend on the choice of the rank-$(r-i)$ flat.

\begin{lemma}\label{lem6} Let $M$ be a binary rank-$(r+1)$ matroid. Then $M/e\cong P_{r}$ for all $e$ in $E(M)$ if and only if $M\cong P_{r+1}\backslash P_{r-i}$ for some $i$ in $\{0,1,\dots, r\}$.
\end{lemma}

\begin{proof} If $M\cong  P_{r+1}\backslash P_{r-i}$ for some $i$ in $\{0,1,\dots, r\}$, then, by, for example, \cite[Corollary 6.2.6]{ox1}, $M/e\cong P_r$ for all $e$ in $E(M)$. 

Now suppose $E(M)\cong P_r$ for each $e$ in $E(M)$. We may assume that $|E(P_{r+1})-E(M)|\geq 2$ otherwise the result certainly holds. Let $x$ and $y$ be distinct elements of $E(P_{r+1})-E(M)$. Then the third element $z$ on the line of $P_{r+1}$ that is spanned by $\{x,y\}$ must also be in $E(P_{r+1})-E(M)$ otherwise $M/z$ is not isomorphic to $P_r$. It follows that, for a basis $X$ of $P_{r+1}\backslash E(M)$, by \cite[Theorem 1]{murty}, $\cl_{P_{r+1}}(X)=E(P_{r+1})-E(M)$. Since $M$ has rank $r+1$, we deduce that $\cl_{P_{r+1}}(X)$ is a projective geometry of rank at most $r$. The lemma follows.
\end{proof}
We omit the straightforward proof of the next result. 

\begin{lemma}\label{lem7} For $r\geq 3$, if $M$ is the binary matroid $P_{r}\backslash P_{r-i}$ for some $i$ in the set $\{1,2,\dots, r-1\}$, then $M$ has a flat isomorphic to either $M(K_4)$ or $U_{3,4}$.
\end{lemma}

For all $q>2$, the only rank-2 members of $\mathcal{M}_q$ are $U_{2,2}$ and $U_{2,q+1}$. Natural obstructions to membership of $\Mq$ are the lines that contain more than two but fewer than $q+1$ points, that is, the collection $\{U_{2,i}:3\leq i \leq q\}$. In rank three, the only matroid that is not in $\Mq$ and has no member of $\{U_{2,i}:3\leq i \leq q\}$ as an induced minor is $U_{3,q+2}$. By Bose \cite{bose}, we note that $U_{3,q+2}$ is representable over $GF(q)$ if and only if $q$ is even. Therefore, the collection $\mathcal{N}=\{U_{2,3},U_{2,4},\dots,U_{2,q},U_{3,q+2}\}$ is contained in the collection of forbidden induced minors for the class $\Mq$. The next result highlights some structure in matroids that have members of $\mathcal{N}$ as induced minors.

\begin{lemma}\label{lem8}
    For some $q>2$, let $M$ be a $GF(q)$-representable matroid having rank at least three. Suppose that $M\not\cong P_{r}$ but that $M/e \cong P_{r-1}$ for all $e$ in $E(M)$. Then $M$ has a member of $\mathcal{N}$ as an induced minor.
\end{lemma}

\begin{proof} Suppose $M$ has no member of $\mathcal{N}$ as an induced minor. Suppose $r(M)=3$. Since the contraction of any element would result in a $(q+1)$-point line, $E(M)$ must have at least $q+2$ elements. Moreover, since $M$ is not $U_{3,q+2}$, we deduce that $M$ contains a triangle $\{p_1,p_2,p_3\}$. This triangle must be contained in a full $(q+1)$-point line of $M$, otherwise $M$ would contain a member of $\mathcal{N}$ as an induced restriction. Label this line $\{p_1,p_2,\dots,p_{q+1}\}$. Since $|E(M)|\geq q+2$, we may choose an element $e$ in $E(M)-\{p_1,p_2,\dots,p_{q+1}\}$. If $e$ is unique, then $M/p_1$ is isomorphic to $U_{2,2}$, a contradiction. Without loss of generality, suppose there is a third point on the line $\cl(\{e,p_1\})$. Then this line is a full $(q+1)$-point line. Therefore, there are two full lines meeting at $p_1$. If this were the entire matroid, then $M/p_1$ consists of only two points. Hence there is an additional element $f$ in $M$ not on either of the lines $\cl(\{e,p_1\})$ or $\{p_1,p_2,\dots,p_{q+1}\}$. Each point of the line $\cl(\{e,p_1\})$ together with $f$ defines a line that meets the line $\{p_1,p_2,\dots,p_{q+1}\}$ in a distinct point and so each of these lines is also full. This gives that every line is full except possibly the line $\cl(\{p_1,f\})$. If any of the points on the line $\cl(\{p_1,f\})$ is absent, then, for some $i$ in $[q+1]$, the line $\cl(\{e,p_i\})$ contains only $q$ points, a contradiction. Therefore,  every line of $M$ is full and $M$ must be a projective geometry, a contradiction. Thus the result holds when $r(M)=3$.

Now suppose the result holds for $r(M)<k$ and let $r(M)=k\geq 4$. Let $e$ be an element of $M$. Since $M/e\cong P_{r-1}$, each of the lines of $P_r$ that contains $e$ must contain a second point of $M$. There must be such a line, say $L$, that contains exactly two points of $M$ otherwise every such line has exactly $q+1$ points and $M\cong P_r$, a contradiction. Let $L=\{e,f\}$. Take a point $g$ in $E(M)-L$ and consider the plane $Q=\cl_M(\{e,f,g\})$. Since $M/h\cong P_r$ for every point $h$ of this plane, it follows that $Q/h\cong U_{2,q+1}$. It follows by the induction assumption that $Q\cong P_3$. Hence $L$ is a $(q+1)$-point line, a contradiction. We conclude, by induction, that the lemma holds. 
\end{proof}

\begin{lemma}\label{lem9} Let $M$ be a $GF(q)$-representable matroid and let $(X,Y)$ be an exact vertical $k$-separation of $M$. Suppose that both $M|\cl_M(X)$ and $M|\cl_M(Y)$ are in $\Mq$. Then either $\cl_M(X)\cap \cl_M(Y)$ is a projective geometry of rank $k-1$, or $M$ has a member of $\mathcal{N}$ as an induced minor.
\end{lemma}

\begin{proof}
    This is immediate if $k=1$ and follows by Lemma \ref{lem4} when $k=2$, so we may assume that $k\geq 3$. Let $M$ be an induced-minor-minimal counterexample and let $r(M)=r$. Suppose first that $\cl_M (X)\cap \cl_M (Y)$ is empty. Then in the green-red coloring of $P_r$ corresponding to $M$, all of the elements of $\cl_{P_r} (X)\cap \cl_{P_r} (Y)$ are red. Take $e$ in $X$ and suppose that $r(X)>k$. Then $M/e$ has an exact vertical $k$-separation $(X',Y')$ corresponding to $(X-e,Y)$. By the minimality of $M$, we deduce that $\cl_{M/e}(X')\cap \cl_{M/e}(Y')$ is a projective geometry of rank $k-1$.

    Since $k\geq 3$, there is a projective line $L$ contained in $\cl_{P_r}(X)\cap \cl_{P_r}(Y)$. In the green-red coloring of $P_r$, every element of $L$ is red. But every element of $L$ is green in the coloring of $P_r/e$. Thus, in $P_r$, for each of the points $z_1,z_2,\dots,z_{q+1}$ of $L$, there is a green point on the line $\cl_{P_r}(\{e,z_i\})$ other than $e$. Thus, when $q=2$, we see that the four green points in $\cl_{P_r}(L\cup e)$ form a 4-circuit, a contradiction as $M|\cl_M(X)$ is $GF(2)$-chordal. When $q>2$, because all of the elements of $L$ are red, the set of points in $\cl_{P_r}(L\cup e)$ contains no line with more than $q$ points. Since $M|\cl_M(X)$ is in $\Mq$, it follows that each line in $\cl_{P_r}(L\cup e)$ that contains at least two green points contains exactly two green points. It follows that $M$ has $U_{3,q+2}$ as an induced restriction, a contradiction.

    We may now assume that $r(X)=k=r(Y)$. Since $(X,Y)$ is an exact $k$-separation, $r(X)+r(Y)-r(M)=k-1$, so $r(M)=k+1$. As $M$ has at least $2k$ elements, it has an element $f$ that is not a coloop. Then the construction of members of $\Mq$ implies that $f$ is on a $(q+1)$-point green line of $M$. This line must meet $\cl_{P_r}(X)\cap \cl_{P_r}(Y)$ so $\cl_{P_r}(X)\cap \cl_{P_r}(Y)$ is not entirely red, a contradiction.

    We conclude $\cl_M (X)\cap \cl_M (Y)$ contains at least one point, say $z$. In $M/z$, there is an exact vertical $(k-1)$-separation $(X'',Y'')$ corresponding to $(X-z,Y-z)$. We deduce, by the minimality of $M$, that $\cl_{P_r/z}(X'')\cap \cl_{P_r/z}(Y'')$ is a projective geometry of rank $k-2$. By Lemmas \ref{lem7} and \ref{lem8}, we deduce that $M|\left(\cl_M (X)\cap \cl_M (Y)\right)$ must have, as an induced minor, a matroid that is not in $\Mq$. 
\end{proof}

\begin{corollary}\label{cor1}
    For all $k\geq 1$, an induced-minor-minimal non-$GF(q)$-chordal matroid has no vertical $k$-separations.
\end{corollary}

A matroid $M$ is \textit{round} if $M$ has no two disjoint cocircuits. Equivalently, $M$ is round if there is no $k$ for which $M$ has a vertical $k$-separation (see, for example, \cite[Lemma 8.6.2]{ox1}). The following result is immediate from the constructive definition of $GF(q)$-chordal matroids.

\begin{lemma}\label{lem13}
    A rank-$r$ matroid $M$ in $\Mq$ is round if and only if $M\cong P_r$.
\end{lemma}

%\section{GF(2)-Chordal Matroids}
The following is a straightforward consequence of the definition.
\begin{lemma}
    A simple binary matroid is chordal if and only if it does not have $U_{3,4}$ as an induced minor.
\end{lemma}

\begin{lemma}
    All $GF(2)$-chordal matroids are chordal matroids.
\end{lemma}

\begin{proof} 
Let $M$ be a $GF(2)$-chordal matroid. Since the class of $GF(2)$-chordal matroids is closed under induced minors, $U_{3,4}$ is not an induced minor of $M$, so $M$ is a chordal matroid. 
\end{proof}

Since $M(K_4)$ is chordal but not $GF(2)$-chordal, it is clear that the class of binary matroids that are chordal properly contains the class of $GF(2)$-chordal matroids.

We now give a common proof of the two main results of the paper.

\begin{proof}[Proof of Theorems  \ref{theo3} and \ref{theo4}]
Let $M$ be an induced-minor-minimal non-$GF(q)$-chordal matroid. By Corollary \ref{cor1}, $M$ is round. By Geelen, Gerards, and Whittle \cite{geelen03}, $M/e$ is also round for all $e$ in $E(M)$. Thus by Lemma \ref{lem13}, $M/e\cong P_{r-1}$ where $r(M)=r$.

First, let $q=2$. Then, by Lemma \ref{lem6}, $M\cong P_r \backslash P_{r-i}$ for some $i$ in $\{1,2,\dots, r-1\}$. Moreover, since $M$ is not a $GF(2)$-chordal matroid, $r\geq 3$. Then, by Lemma \ref{lem7}, $M$ has $M(K_4)$ or $U_{3,4}$ as an induced restriction. As each of these matroids is an induced-minor-minimal non-$GF(2)$-chordal matroid, we deduce that $M$ is isomorphic to $M(K_4)$ or $U_{3,4}$.

Now assume that $q>2$. Since each member of $\mathcal{N}$ is an induced-minor-minimal non-$GF(q)$-chordal matroid, we may assume that $M$ has no member of $\mathcal{N}$ as an induced minor. Thus $r(M)\geq 3$. By Lemma \ref{lem8}, we get a contradiction. 
\end{proof}

% Suppose $M$ is an induced-minor-minimal non-$GF(2)$-chordal matroid that is not $M(K_4)$ or $U_{3,4}$. Then by Corollary \ref{cor1} the matroid $M$ has no vertical separations and is therefore round, a contradiction to Lemma \ref{lem212}.
% \end{proof}

%\section{$GF(q)$-chordal Matroids}

% When $q>2$, we remark that the class $\mathcal{N}=\{U_{2,3},U_{2,3},\dots, U_{2,q},U_{3,q+2}\}$ is clearly contained in the class of all induced-minor-minimal non-$GF(q)$-chordal matroids. 

% \begin{lemma}
%     Suppose $M$ is an induced-minor-minimal non-$GF(q)$-chordal matroid that is not a member of $\mathcal{N}$. Then, for each element $e$ of $M$, the matroid $\si(M/e)$ is $3$-connected.
% \end{lemma}

% \begin{proof}
%     This is an immediate consequence of Corollary \ref{cor1}.
%     \end{proof}

% \begin{proof}[Proof of Theorem \ref{theo4}] 
% Let $M$ be an induced-minor-minimal non-$GF(q)$-chordal matroid that is not a member of $\mathcal{N}$. Then by Corollary \ref{cor1}, the matroid $M$ has no vertical separations and is therefore round, a contradiction to Lemma \ref{lem212}  
% \end{proof}

\section{Dirac's Other Characterization}

In \cite{dirac}, another characterization is given of chordal graphs. In a graph $G$, a \textit{vertex separator} is a set of vertices whose deletion produces a graph with more connected components than $G$. 

\begin{theorem}\label{theo6}
    A graph is chordal if and only if every minimal vertex separator induces a clique.
\end{theorem}

It is shown in Lemma \ref{lem9} that, if $M$ is a $GF(q)$-chordal matroid, then, for every exact vertical $k$-separation $(X,Y)$ of $M$, the restriction $M|(\cl(X) \cap \cl(Y))\cong P_{k-1}$. However, the converse of this is not true. For example, the matroid $P_{U_{2,3}}(M(K_4),M(K_4))$ is not $GF(2)$-chordal, but the only exact vertical $k$-separation has $k=3$ and has $U_{2,3}$ as the intersection of the closures of the two sides of the vertical $3$-separation. A \textit{divider} in a matroid is an exact vertical $k$-separation for some $k$. A divider $(X,Y)$ is \textit{minimal} if there is no vertical $k'$-separation $(X',Y')$ such that $\cl(X')\cap \cl(Y')\subsetneqq \cl(X)\cap \cl(Y)$. Recall that, for sets $X$ and $Y$ in a matroid $M$, the \textit{local connectivity} between $X$ and $Y$, denoted $\sqcap(X,Y)$ or $\sqcap_M(X,Y)$, is given by $\sqcap(X,Y)=r(X)+r(Y)-r(X\cup Y)$. Let $\mathcal{N}_q$ be the class of $GF(q)$-representable matroids $N$ such that, for every minimal divider $(X,Y)$ of $N$, the matroid $N|(\cl(X)\cap \cl(Y))$ is a projective geometry of rank $\sqcap(X,Y)$. Since round matroids have no vertical $k$-separations, all $GF(q)$-round matroids are in $\mathcal{N}_q$.

\begin{lemma}\label{lem20}
    Let $(X,Y)$ be a minimal divider of a matroid $N$. Then $\cl(X)\cap \cl(Y)=\cl(X)\cap \cl(Y-\cl(X))$.
\end{lemma}
\begin{proof}
    Suppose $y\in \cl(X)\cap Y$. Then $r(Y)=r(Y-y)$; otherwise, $y$ is a coloop of $N|Y$ and $(X\cup y, Y-y)$ is a divider with $\cl(X\cup y)\cap \cl(Y-y)$ properly contained in $\cl(X)\cap \cl(Y)$, a contradiction. We deduce that $\cl(Y)=\cl(Y-\cl(X))$.
\end{proof}
%A rank-$(k-1)$ flat $F$ in a matroid $M$ is a \textit{divider} if $M$ has an exact vertical $k$-separation $(X,Y)$ for some $k\geq 2$ such that $F=\cl(X)\cap \cl(Y)$. A divider $F$ is a \textit{minimal divider} if no proper subset of $F$ is a divider. Let $\mathcal{N}_q$ be the class of $GF(q)$-representable matroids $N$ such that, for every minimal divider $F$ of $N$, the matroid $N|F$ is a projective geometry. Since round matroids have no vertical $k$-separations, all round matroids are in $\mathcal{N}_q$.  

%\begin{lemma}\label{lem20}
   % If $F$ is a minimal divider of a matroid $M$ corresponding to an exact vertical $k$-separation $(X,Y)$, then $(\cl(X),Y-\cl(X))$ is also an exact vertical $k$-separation having $F=\cl(X)\cap \cl(Y-\cl(X))$.
%\end{lemma}
% \begin{proof}
%     Let $y$ be an element of $\cl(X)\cap Y$. Then $r(Y)=r(Y-y)$; otherwise, $y$ is a coloop of $Y$ and $(X\cup y, Y-y)$ is an exact vertical $k'$-separation with  $\cl(X)\cap \cl(Y-y)$ properly contained in $F$, a contradiction. Therefore, $\cl(Y)=\cl(Y-\cl(X))$.
% \end{proof}

The next theorem is an analog of Theorem \ref{theo6}.

\begin{theorem}
    A matroid $M$ is in $\mathcal{N}_q$ if and only if $M$ can be constructed from round $GF(q)$-representable matroids by a sequence of generalized parallel connections across projective geometries.
\end{theorem}
\begin{proof}
    Let $\mathcal{R}_q$ be the class of matroids that can be constructed from round $GF(q)$-representable matroids by a sequence of generalized parallel connections across projective geometries. It suffices to prove that a connected matroid $M\in\mathcal{N}_q$ if and only if it is in $\mathcal{R}_q$. Suppose $M\in\mathcal{N}_q$. If $M$ has no dividers, then $M$ is round, so $M\in\mathcal{R}_q$. 
    Hence, we may assume $M$ has a minimal divider $(X,Y)$. Then $M|(\cl(X)\cap \cl(Y))$ is a projective geometry $N$ of rank $\sqcap(X,Y)$, and $M=P_N(M|\cl(X),M|\cl(Y))$. Letting $M_X=M|\cl(X)$, we see that $M_X$ is either round or has a minimal divider $(U,V)$, where $M_X|(\cl(U)\cap \cl(V))$ is a projective geometry, $N'$, of rank $\sqcap_{M_X}(U,V)$, and $M_X$ is equal to $P_{N'}(M_X|\cl(U),M_X|\cl(V))$. Continuing in this way, we see that every matroid in $\mathcal{N}_q$ can be obtained in the manner prescribed. Hence $\mathcal{N}_q\subseteq \mathcal{R}_q$.

    We will prove that $\mathcal{R}_q\subseteq \mathcal{N}_q$ by induction on the number $n$ of round matroids used to construct a connected member $M$ of $\mathcal{R}_q$. If $n=1$, then $M$ is round and so $M$ is in $\mathcal{N}_q$. Now suppose that the result holds when $n\leq t-1$ and assume that $M$ is constructed by using exactly $t$ round matroids. Then $M\cong P_N(M_1,M_2)$, where $M_2$ is a round matroid and $N$ is a projective geometry. Let $(X,Y)$ be a minimal divider of $M$ and let $F=\cl(X)\cap \cl(Y)$. We need to show that $M|F\cong P_k$ where $r(F)=k$. Let $X_N=X\cap E(N)$ and $Y_N=Y\cap E(N)$. Also let $X_i=(X\cap E(M_i))-X_N$ and $Y_i=(Y\cap E(M_i))-Y_N$ for each $i$ in $\{1,2\}$. Since $N$ is a projective geometry, we may suppose that $X_N$ spans $Y_N$. Therefore $\cl(Y_N)\subseteq F$. As $M_2$ is round and has $(X_N\cup Y_N\cup X_2,Y_2)$ as a partition, either $X_N\cup X_2$ spans $Y_2$, or $Y_2$ spans $M_2$. In the latter case, $Y$ spans $E(N)$, so $E(N)\subseteq F$. Now, $(E(M_1),E(M_2)-E(N))$ is a divider of $M$ and $\cl(E(M_1))\cap \cl(E(M_2)-E(N))=E(N)$. Because $(X,Y)$ is a minimal divider, we have $F=E(N)$, so $M|F\cong P_k$ where $r(F)=k$. 
    
    We deduce that $X_N\cup X_2$ spans $Y_2$. Then $\cl(X)$ contains $E(M_2)$ and, by Lemma~\ref{lem20}, we may assume that $Y\subseteq E(M_1)-E(N)$. Thus 
    \begin{equation}\label{eq1}
        F=\cl(X)\cap \cl(Y)=\cl_{M_1}(X\cap E(M_1))\cap \cl_{M_1}(Y).
    \end{equation}

    We show next that $(X\cap E(M_1),Y)$ is a minimal divider of $M_1$. It is a divider of $M_1$ unless $X\cap E(M_1)$ or $Y$ spans $M_1$. 
    In the exceptional case, as $X$ does not span $M$, we see that $X\cap E(M_1)$ does not span $M_1$. Thus $Y$ spans $M_1$, so $E(N)\subseteq F$. 
    Hence $E(N)=F$, and $M|F\cong P_k$ as desired. Thus $(X\cap E(M_1),Y)$ is a divider of $M_1$. Suppose it is not minimal. 
    Then, by (\ref{eq1}), $M_1$ has a minimal divider $(X_1,Y_1)$ such that $\cl_{M_1}(X_1)\cap \cl_{M_1}(Y_1)\subsetneqq F$. Now we may assume that $\cl_{M_1}(X_1)\supseteq E(N)$, so $(\cl_{M_1}(X_1),Y_1-\cl_{M_1}(X_1))$ is a minimal divider of $M_1$. It follows that $(E(M_2)\cup\cl_{M_1}(X_1))\cap \cl_M(Y_1-\cl_{M_1}(X_1))=\cl_{M_1}(X_1)\cap \cl_{M_1}(Y_1)\subsetneqq F$, a contradiction. Hence $(X\cap E(M_1),Y)$ is a minimal divider of $M_1$. By the induction assumption, $M_1|F\cong P_k$, so $M|F\cong P_k$, as desired.
\end{proof}

\end{document}